\documentclass[11pt]{article} 

\usepackage[vmargin=2.5cm,hmargin=3cm]{geometry}

\usepackage[english]{babel}
\usepackage[latin1]{inputenc}
\usepackage{amsmath}
\usepackage{amsthm}
\usepackage{amssymb}
\usepackage{mathrsfs}
\usepackage{tikz-cd}
\usepackage{caption}
\usepackage{comment}

\usepackage{hyperref}
\hypersetup{
	colorlinks   = true,                    
	linkcolor    = blue!70!black,
	anchorcolor  = gray,
	citecolor    = blue!70!black,
	filecolor    = red,
	menucolor    = green,
	runcolor     = red,
	urlcolor     = [rgb]{0,0.2,0.5}
}

\newtheorem{theorem}{Theorem}[section]
\newtheorem{proposition}[theorem]{Proposition}
\newtheorem{corollary}[theorem]{Corollary}
\newtheorem{lemma}[theorem]{Lemma}
\newtheorem{conjecture}[theorem]{Conjecture}

\theoremstyle{definition}

\newtheorem{definition}[theorem]{Definition}
\newtheorem{remark}[theorem]{Remark}
\newtheorem{example}[theorem]{Example}

\DeclareMathOperator{\Ap}{\operatorname{Ap}}

\title{Cyclotomic numerical semigroup polynomials
	with at most two irreducible factors}
\author{Alessio Borz\`{i}, Andr\'{e}s Herrera-Poyatos, Pieter Moree}

\begin{document}
	
	\maketitle
	
	\begin{abstract}
		A numerical semigroup $S$ is cyclotomic if its semigroup polynomial $\mathrm{P}_S$ is a product of cyclotomic polynomials. The number of irreducible factors of $\mathrm{P}_S$ (with multiplicity) is the polynomial length $\ell(S)$ of $S.$
		We show that a cyclotomic numerical semigroup is complete intersection if 
		$\ell(S)\le 2$. This establishes a particular case of a conjecture of Ciolan, Garc\'{i}a-S\'{a}nchez and Moree (2016) claiming that every cyclotomic numerical semigroup is complete intersection. In addition, we investigate the relation between $\ell(S)$ and the embedding dimension of $S.$
	\end{abstract}
	
	\section{Introduction}
	
	The $n$-th cyclotomic polynomial is the minimal polynomial of any primitive $n$-th root of unity
	\begin{equation}\label{cyclotomic polynomial}
	\Phi_n(x) = \prod_{\substack{ j=1 \\ (j,n) = 1 }}^n \left( x - e^{2 \pi i j / n} \right) = \sum_{k=0}^{\varphi(n)} a_n(k) x^k,
	\end{equation}
	where $\varphi$ is Euler's totient function. It is of degree $\varphi(n)$ and has integer coefficients.
	
	Let $\mathbb{N}$ denote the set of non-negative integers. A \emph{numerical semigroup} $S$ is an additive submonoid of $\mathbb{N}$ with finite complement in $\mathbb{N}$. The \emph{semigroup polynomial} of $S$ is defined by $\mathrm{P}_S(x) = 1 + (x-1) \sum_{g \in \mathbb{N} \setminus S} x^g$.
	If $p \neq q$ are primes and $\langle p,q \rangle$ is the numerical semigroup generated by $p$ and $q$, then
	\begin{equation}\label{folklore result}
	\Phi_{pq}= \mathrm{P}_{\langle p,q \rangle},
	\end{equation}
	see for instance \cite{moree2014numerical}.
	This identity can be used to reprove various properties of cyclotomic polynomials, e.g., that $a_{pq}(k) \in \{ 0,1,-1 \}$, a result due to Migotti  \cite{migotti1883theorie}.
	More generally, if $p$ and $q$ are two coprime non-negative integers, then $\mathrm{P}_{\langle p,q \rangle}$ is a product of cyclotomic polynomials. 
	There are various other interesting infinite families of numerical semigroups 
	such that their semigroup polynomial has only cyclotomic factors.
	These facts led Ciolan, Garc\'{i}a-S\'{a}nchez and Moree \cite{moree2014numerical} to define 
	\emph{cyclotomic numerical semigroups} as numerical semigroups whose semigroup polynomial is a product of cyclotomic polynomials. 
	For this family of numerical semigroups it is easy to see that the following implications hold:
	\begin{equation}\label{implications}
	\text{complete intersection} \Longrightarrow \text{cyclotomic} \Longrightarrow \text{symmetric}
	\end{equation}
	(see Section~\ref{Section 2.4}). The converse of the second implication of \eqref{implications} is far from true. In fact, for any odd integer $F$ with $F \ge 9$, there is a numerical semigroup with Frobenius number $F$ that is symmetric and non-cyclotomic. This was proven independently by Garc{\'\i}a-S{\'a}nchez (in the appendix of  \cite{herrera2018coefficients}), 
	Herrera-Poyatos and Moree \cite{herrera2018coefficients} and Sawhney and Stoner \cite{sawhney2018symmetric}.
	In the latter two papers it is shown (by quite different
	methods) that, for every $k \geq 5$, the polynomial
	\[ 1-x+x^k-x^{2k-1}+x^{2k}, \]
	which is the semigroup polynomial of the symmetric numerical semigroup $S_k = \langle k,k+1,\dots,2k-2 \rangle$, is not a product of only cyclotomic polynomials. Therefore, $S_k$ is symmetric, but not cyclotomic.
	It was conjectured by Ciolan, Garc\'{i}a-S\'{a}nchez and Moree \cite{ciolan2016cyclotomic} that the converse of the first implication in (\ref{implications}) holds true, more precisely they made the following conjecture.
	
	\begin{conjecture}{\rm \cite[Conjecture 1]{ciolan2016cyclotomic}.}\label{conjecture}
		A numerical semigroup $S$ is cyclotomic if and only if it is complete intersection.
	\end{conjecture}
	
	Using the GAP package \cite{delgadonumericalsgps} the authors of  \cite{ciolan2016cyclotomic} verified that Conjecture~\ref{conjecture} holds true for numerical semigroups with Frobenius number up to $70$. Further, in the context of graded algebras, Borz{\`\i} and D'Al{\`\i} \cite{borzi2020graded} prove a version of Conjecture~\ref{conjecture} for Koszul algebras and for graded algebras that have an irreducible $h$-polynomial.
	
	In this paper we classify all cyclotomic numerical semigroups such that their semigroup polynomial has at most two irreducible polynomial factors.

	\begin{theorem} \label{thm:main}
		Let $S$ be a cyclotomic numerical semigroup.
		\begin{enumerate}
			\item \label{item:main:1} If $\mathrm{P}_S$ is irreducible, then $S=\langle p,q\rangle$ with $p \neq q$ primes and $\mathrm{P}_S = \Phi_{pq}$.
			\item \label{item:main:2} If $\mathrm{P}_S$ is a product of two irreducible polynomials, then either
			\begin{enumerate}
				\item  $S=\langle p,q^2\rangle$ with $p,q$ distinct primes and $\mathrm{P}_S= \Phi_{pq}\Phi_{pq^2}$; or
				\item $S=\langle p, q^2,qr\rangle$ with $p,q,r$ distinct primes such that $p \in \langle q, r \rangle$ and $\mathrm{P}_S= \Phi_{pq} \Phi_{q^2r}$.
			\end{enumerate}
		\end{enumerate}
	\end{theorem}
	As a byproduct, 
	since the numerical semigroups obtained in Theorem~\ref{thm:main} are easily seen to be complete intersections with the help of gluings (see Section~\ref{sec:preliminaries}), Conjecture~\ref{conjecture} holds true in the cases studied in Theorem~\ref{thm:main}.
	\begin{theorem}\label{l leq 2}
		Suppose that the semigroup polynomial $\mathrm{P}_S$ has at most two irreducible factors. Then $S$ is cyclotomic if and only if it is complete intersection.
	\end{theorem}
	
	Theorem~\ref{thm:main} motivates the following definition. We define the \emph{polynomial length} $\ell(S)$ of $S$ as the number of irreducible factors of $\mathrm{P}_S$ (with multiplicity). We study this quantity in Section~\ref{sec:length}.
	
	Our paper is organised as follows. In Section~\ref{sec:preliminaries} we gather some preliminary material that is partly expository and will be useful further on. In Section~\ref{sec:length:1} we prove 
	part~\ref{item:main:1} of Theorem~\ref{thm:main}, and in Section~\ref{sec:length:2} we prove part~\ref{item:main:2}. Finally, in Section~\ref{sec:open-problems} we pose some conjectures involving the polynomial length of cyclotomic numerical semigroups.
	
	The computer algebra computations in this paper were done by using Macaulay2 \cite{M2}, the GAP system \cite{gap2015gap} and, in particular, the NumericalSgps package \cite{delgadonumericalsgps}. 
	
	\section{Preliminaries} \label{sec:preliminaries}
	
	\subsection{Numerical semigroups and Hilbert series}\label{Section 2.1}
	
	For an introduction to numerical semigroups, see \cite{rosales2009numerical}. 
	
	Let $S$ be a numerical semigroup. The \emph{embedding dimension} $\mathrm{e}(S)$ of $S$ is the cardinality of the (unique) minimal generating system of $S$. The \emph{Frobenius number} of $S$ is $\mathrm{F}(S) = \max(\mathbb{Z} \setminus S)$. 
	For example, if $S = \langle a,b \rangle$ with $b>a>1$ coprime integers, then $\mathrm{F}(S)+1 = (a-1)(b-1)$ (see for instance \cite[Proposition 2.13]{rosales2009numerical}).	
	This fact in combination with the observation that if $T \subseteq S$ are two numerical semigroups, then clearly $\mathrm{F}(S) \leq \mathrm{F}(T)$, proves the following lemma.
	
	\begin{lemma}\label{Frobenius upperbound}
		If $a,b \in S$ with $\gcd(a,b)=1$, then $\mathrm{F}(S)+1 \leq (a-1)(b-1)$.
	\end{lemma}	
	
	The \emph{Hilbert series} of $S$ is
	\[ \mathrm{H}_S(x) = \sum_{s \in S} x^s \in \mathbb{Z}[\![x]\!], \]
	and the \emph{semigroup polynomial} of $S$ is 
	\begin{equation}
	\label{flauw}
	\mathrm{P}_S(x) = (1-x)\mathrm{H}_S(x) = 1 + (x-1) \sum_{g \in \mathbb{N} \setminus S}x^g. 
	\end{equation}
	Note that $\deg \mathrm{P}_S= \mathrm{F}(S)+1$.
	The second equality in \eqref{flauw} easily follows from $\mathrm{H}_S(x) + \sum_{g \in \mathbb{N} \setminus S}x^g = 1/(1-x),$ where here and in the sequel we work in $\mathbb{Z}[\![x]\!]$ and use the shorthand $1/(1-x)$
	for $1+x+x^2+x^3+\cdots.$
	
	Some properties of a numerical semigroup, such as symmetry, can be captured in 
	terms of its semigroup polynomial. We need some notation in order to characterize symmetry in this way.	A numerical semigroup $S$ is \emph{symmetric} if for every integer $n$ we have that $n \in S$ if and only if $\mathrm{F}(S)-n \notin S$. A polynomial $f(x) = \sum_{i=0}^d \alpha_i x^i \in \mathbb{Z}[x]$ is \emph{palindromic} (or \emph{self-reciprocal}) if $f(x) = x^d f(x^{-1})$, that is, its coefficients $\alpha_i$ satisfy the relation $\alpha_i=\alpha_{d-i}$.
	
	\begin{theorem}{\rm \cite[Theorem 5]{moree2014numerical}.}\label{symmetrici iff palindromic}
		A numerical semigroup $S$ is symmetric if and only if $\mathrm{P}_S$ is palindromic.
	\end{theorem}

	\subsection{Complete intersection numerical semigroups} \label{sec:pre:gluings}
	
	Let $S_1,S_2$ and $S$ be numerical semigroups, and let $a_1 \in S_2$ and $a_2 \in S_1$ be coprime integers such that they are not minimal generators of their respective semigroups. We say that $S$ is a \emph{gluing} of $S_1$ and $S_2$ at $a_1a_2$ if $S=a_1 S_1 + a_2 S_2$ and we write $S = a_1S_1 +_{a_1a_2} a_2S_2$. We will keep this notation until the end of this section. By \cite[Lemma 9.8]{rosales2009numerical} we have that
	\begin{equation}\label{emdim gluing}
	\mathrm{e}(S) = \mathrm{e}(S_1) + \mathrm{e}(S_2).
	\end{equation}
	
	Let $M$ be a submonoid of $\mathbb{N}$, and let $m \in M \setminus \{0\}$. The \emph{Ap\'{e}ry set} of $M$ at $m$ is the set
	\[ \Ap(M,m) = \{ v \in M : v-m \notin M \}. \]
	Gluings can be characterised in terms of Ap\'{e}ry sets as in Lemma~\ref{Apery and gluing}.
	\begin{lemma}{\rm \cite[Theorem 9.2]{rosales2009numerical}.}\label{Apery and gluing}
		The numerical semigroup $S$ is the gluing of $S_1$ and $S_2$ at $m = a_1a_2$, if and only if the map
		\[ \Ap(a_1 S_1, m) \times \Ap(a_2 S_2, m) \rightarrow \Ap(S,m) \]
		given by $(v,w) \mapsto v+w$ is bijective. If this is the case, then we have $\Ap(S,m) = \Ap(a_1 S_1, m) + \Ap(a_2 S_2, m)$.
	\end{lemma}
	
	Gluings can also be characterized in terms of semigroup polynomials on using  
	\begin{equation} \label{eq:apery-hilbert}
	\sum_{w \in \Ap(S,m)} x^w = \left( 1-x^m \right) \mathrm{H}_S(x),
	\end{equation}
	which follows from $S = \Ap(S,m) + m \mathbb{N}$ \cite[Lemma 2.6]{rosales2009numerical}. 
	
	\begin{proposition}[{\cite[Corollary 4.4]{assi2015frobenius}}] \label{gluing iff Hilbert}
		The following statements are equivalent:
		\begin{enumerate}
			\item $S = a_1 S_1 +_{a_1a_2} a_2 S_2$;
			\item $\mathrm{H}_S(x) = (1-x^{a_1a_2})\mathrm{H}_{S_1}(x^{a_1})\mathrm{H}_{S_2}(x^{a_2})$;
			\item $\mathrm{P}_S(x) = \mathrm{P}_{\langle a_1,a_2 \rangle}(x)\mathrm{P}_{S_1}(x^{a_1})\mathrm{P}_{S_2}(x^{a_2})$.
		\end{enumerate}
	\end{proposition}
	\begin{proof}
		The equivalence of $(2)$ and $(3)$ is trivial, so we just need to prove the equivalence of $(1)$ and $(2)$. From 
		Lemma~\ref{Apery and gluing} we infer that $S = a_1 S_1 +_{a_1a_2} a_2 S_2$ if and only if
		\[ \begin{split}
		\sum_{w \in \Ap(S,a_1a_2)} x^w = \sum_{w_1 \in \Ap(a_1S_1,a_1a_2)}\sum_{w_2 \in \Ap(a_2S_2,a_1a_2)} x^{w_1+w_2} = \\
		= \left( \sum_{w_1 \in \Ap(S_1,a_2)} x^{a_1w_1} \right) \left( \sum_{w_2 \in \Ap(S_2,a_1)} x^{a_2w_2} \right).
		\end{split} \]
		Now apply \eqref{eq:apery-hilbert} and divide both sides by 
		$1-x^{a_1a_2}.$ 
	\end{proof}
	
	Complete intersection numerical semigroups are usually introduced in the context of minimal presentations \cite[Chapter 8]{rosales2009numerical}, where a numerical semigroup $S$ is said to be \emph{complete intersection} if the cardinality of a minimal presentation of $S$ is equal to $\mathrm{e}(S)-1$. In this paper we will only need the recursive characterisation of complete intersection numerical semigroups in terms of gluings. 
	
	\begin{theorem}{\rm \cite[Theorem 9.10]{rosales2009numerical}.}\label{gluing ci}
		A numerical semigroup is complete intersection if and only if $S$ is $\mathbb{N}$ or $S$ is a gluing of two complete intersection numerical semigroups.
	\end{theorem}
	
	Using this characterisation in terms of gluings along with Proposition~\ref{gluing iff Hilbert} one can determine the Hilbert series of $S$ as follows.
	
	\begin{corollary}[{\cite[Theorem 4.8]{assi2015frobenius}}]\label{complete intersection hilbert series}
		Let $S = \langle n_1,\dots,n_e \rangle$ be a complete intersection numerical semigroup. Then there are $d_1, \ldots, d_{e-1} \in S \setminus \{n_1, \ldots, n_e\}$ such that
		\[ \mathrm{H}_S(x) = \frac{(1-x^{d_1}) \dots (1-x^{d_{e-1}})}{(1-x^{n_1}) \dots (1-x^{n_e})}. \]
	\end{corollary}
	
	The integers $d_i$ in the previous result are actually the \emph{Betti elements} of $S$ (with multiplicity). We refer to \cite[Chapter 7]{rosales2009numerical} for a definition of Betti element, which will be needed only in Section \ref{sec:open-problems}. If $\mathrm{e}(S)=2$, so $S = \langle a,b \rangle = a \mathbb{N} +_{ab} b \mathbb{N}$, then Corollary~\ref{complete intersection hilbert series} and the fact that $\deg \mathrm{P}_S = \mathrm{F}(S) = (a-1)(b-1)$ yield
	\begin{equation}\label{embedis2}
	\mathrm{P}_{\langle a,b\rangle}(x)=\frac{(1-x)(1-x^{ab})}{(1-x^a)(1-x^b)}.
	\end{equation}
	In the case when $\mathrm{e}(S) = 3$, we have the following result due to Herzog.
	
	\begin{theorem}{\rm \cite[Theorem 4.2.1]{herzog1970generators}.} \label{thm:e3}
		Let $S$ be a numerical semigroup. If $\mathrm{e}(S) = 3$,  then $S$ is complete intersection if and only if it is symmetric.
	\end{theorem}
	
	\subsection{Cyclotomic polynomials}
	\label{subsec:cyclotomic}
	
	For an introduction to cyclotomic polynomials, see \cite{thangadurai2000coefficients}.

	From the definition (\ref{cyclotomic polynomial}) we have
	\begin{equation}\label{xn-1}
	x^n-1 = \prod_{d \mid n} \Phi_d(x).
	\end{equation}
	This in combination with \eqref{embedis2} yields for example that
	\begin{equation}
	\label{e=2factorincyclotomics}
	\mathrm{P}_{\langle a,b\rangle}=\prod_{\substack{d|ab,~d\nmid a,~d\nmid b}}\Phi_{d}.
	\end{equation}
	An important property of the cyclotomic polynomials is that they are irreducible over the rationals, several
	famous mathematicians gave different proofs of this, cf. Weintraub \cite{Wein}. Hence, \eqref{xn-1} gives the factorization
	of $x^n-1$ into irreducibles.
	
	By the so called \emph{M\"{o}bius inversion formula} (see \cite[Proposition 3.7.1]{stanley2011enumerative}) we
	infer from \eqref{xn-1} that
	\begin{equation}\label{Mobius inversion formula}
	\Phi_n(x) = \prod_{d \mid n} (x^d-1)^{\mu(n/d)},
	\end{equation}
	where the \emph{M\"{o}bius function} $\mu$ is defined by
	\[ \mu(n) = \begin{cases}
	1 & n=1;\\
	(-1)^k & n = p_1 \cdots p_k \text{ with the }p_i \text{ distinct primes};\\
	0 & \text{otherwise}.
	\end{cases} \]
	By taking degrees in \eqref{Mobius inversion formula} we obtain $\varphi(n)=\sum_{d|n}d\mu(n/d)$.
	If $n > 1$, then $\sum_{d \mid n} \mu(n/d) = 0$, so equation (\ref{Mobius inversion formula}) can be rewritten as
	\begin{equation}\label{Mobius inversion formula2}
	\Phi_n(x) = \prod_{d \mid n} (1-x^d)^{\mu(n/d)}.
	\end{equation}
	Recall that a polynomial $f$ of degree $d$ is palindromic if $f(x) = x^d f(x^{-1})$. 
	As, for $n>1$,
	\begin{equation} 
	\label{eq:reci}
	x^{\varphi(n)}\Phi_n \left( \frac{1}{x} \right) = x^{\sum_{d|n}d\mu(n/d)}\prod_{d \mid n} \left( 1-\frac{1}{x^d} \right)^{\mu(n/d)}=\prod_{d \mid n} (x^d-1)^{\mu(n/d)}=\Phi_n(x),
	\end{equation}
	we see that $\Phi_n(x)$ is palindromic for $n>1$.
	
	It follows from \eqref{Mobius inversion formula2} that for $n>1$ we have $\Phi_n(0)=1$ and
	\begin{equation} 
	\label{eq:first-cyclo-coeff}
	\Phi_n(x) \equiv 1-\mu(n)x\,\,({\rm mod ~}x^2).
	\end{equation}
	Let $r$ be any natural number. Using \eqref{Mobius inversion formula} and the fact that the M\"obius function is only supported on squarefree integers, we see that
	$$
	\Phi_{nr^2}(x) = \prod_{d \mid nr^2} (x^{nr^2/d}-1)^{\mu(d)}=
	\prod_{d \mid nr} ((x^r)^{nr/d}-1)^{\mu(d)},
	$$
	and so
	\begin{equation}\label{eq:cyc property}
	\Phi_{nr^2}(x) = \Phi_{nr}(x^r).
	\end{equation}

	\subsection{Cyclotomic numerical semigroups}\label{Section 2.4}
	
	A numerical semigroup $S$ is \emph{cyclotomic} if $\mathrm{P}_S$ is a product of cyclotomic polynomials, that is, $\mathrm{P}_S=\prod_{d \in \mathcal{D}} \Phi_d^{f_d}$, for some finite set $\mathcal{D}$, with 
	$f_d\ge 1$. As $\mathrm{P}_S(1)=1,$ $\Phi_1$ does not appear in this product. \\
	\indent Corollary~\ref{complete intersection hilbert series} gives rise to the following result of Ciolan et al.\,\cite{ciolan2016cyclotomic}.

	\begin{corollary}\label{complete intersection is cyclotomic}
		Every complete intersection numerical semigroup is cyclotomic.
	\end{corollary}
	\begin{proof}
		By Corollary~\ref{complete intersection hilbert series} and \eqref{xn-1} we have $\mathrm{P}_S(x)=(1-x)\mathrm{H}_S(x)=\prod \Phi_n(x)^{f_n},$ with possibly $f_n<0.$
		However, this would imply that $\mathrm{P}_S(x)$ has a pole at $x=e^{2\pi i/n},$ contradicting the fact that $\mathrm{P}_S$ is a polynomial.
	\end{proof}

	Observe that the product of two palindromic polynomials is palindromic. Hence, by applying Theorem~\ref{symmetrici iff palindromic} and recalling that $\Phi_n$ is palindromic for 
	$n>1$ (see Section~\ref{subsec:cyclotomic}), we reach the following conclusions. 
	
	\begin{corollary}[{\cite[Theorem 1]{ciolan2016cyclotomic}}] \label{cor:symmetric}
		Every cyclotomic numerical semigroup is symmetric.
	\end{corollary}

	\begin{corollary}[{\cite[Lemma 7]{ciolan2016cyclotomic}}] 
		\label{lessthanfour}
		Conjecture~\ref{conjecture} holds true for those $S$ with $\mathrm{e}(S)\le 3$.
	\end{corollary}
	\begin{proof}
		This follows from Corollary~\ref{cor:symmetric} and Theorem~\ref{thm:e3}.
	\end{proof}
	To conclude this section, we note that the cyclotomicity of numerical semigroups is preserved under gluing.
	\begin{corollary}\label{gluing cyclotomic}
		If $S$ is the gluing of $S_1$ and $S_2$ at $a_1a_2$, then $S$ is cyclotomic if and only if both $S_1$ and $S_2$ are cyclotomic.
	\end{corollary}
	\begin{proof}
		This follows from Proposition~\ref{gluing iff Hilbert}.
	\end{proof}
	
	\section{Cyclotomic numerical semigroups of polynomial length 1} \label{sec:length:1}
	
	In this section we classify cyclotomic numerical semigroups having irreducible semigroup polynomial. They
	are given in Corollary~\ref{onefactorclassification}.
	
	\begin{lemma} \label{lem:twoprimes}
		Let $S$ be a numerical semigroup such that $\Phi_n$ divides $\mathrm{P}_S$ for some $n$.  If $p, q \in S$ are two different primes dividing $n$, then $S = \langle p,q\rangle$ and $\mathrm{P}_S = \Phi_{pq}$. 
	\end{lemma}
	\begin{proof}
		Recall that $\deg \Phi_n = \varphi(n)$, see \eqref{cyclotomic polynomial}, and $\deg \mathrm{P}_S =  \mathrm{F}(S)+1$. Hence, 
		\[(p-1)(q-1) \le \varphi(n) \le \deg \mathrm{P}_S =  \mathrm{F}(S)+1.\]
		By Lemma~\ref{Frobenius upperbound}, with $a = p$ and $b = q$, it follows that
		\[(p-1)(q-1) \le \varphi(n) \le \mathrm{F}(S)+1 \leq (p-1) (q-1),\]
		and we conclude that $\mathrm{P}_S = \Phi_n$ and $\varphi(n) = \varphi(p q)$. 
		Writing $n=kpq$ we have 	$\varphi(pq)=\varphi(n)=\varphi(kpq)\ge \varphi(k)\varphi(pq)$ and hence $\varphi(k)=1,$ implying $k=1$ or $k=2$.
		Consequently, $n = pq$ or $n = 2pq$. From \eqref{eq:first-cyclo-coeff} and the equality 
		$\mathrm{P}_S(x) \equiv \Phi_{n}(x)\,\,({\rm mod ~}x^2),$ we obtain $\mu(n) = 1$. Therefore $n = pq$ and $\mathrm{P}_S = \Phi_{p q}$. By \eqref{folklore result} we conclude that $S =\langle p, q\rangle $.
	\end{proof}
	
	\begin{theorem} \label{thm:l1}
		Let $S$ be a numerical semigroup. Let $j,h$ and
		$n$ be positive integers such that $\mathrm{P}_S(x)=\Phi_n(x^j)^h$. 
		Then $S=\langle p,q\rangle$ with $p \neq q$ primes, $n=pq$ and $j=h=1.$
	\end{theorem}
	\begin{proof} 
		On the one hand we have $\mathrm{P}_S(x) \equiv 1 - x\,({\rm mod~}x^2)$, and on the other $\Phi_n(x^j)^h\equiv 1 - \mu(n)hx^j\,({\rm mod~}x^{2j}).$ We conclude that $\mu(n)=j=h=1$, so $\mathrm{P}_S(x) = \Phi_n(x)$. From $1=\mathrm{P}_S(1)=\Phi_n(1)$ we infer $n\ge 2$.
		Thus $n$ is a product of an even number of distinct primes $p_1 < p_2 < \dots < p_{2k},$ that is $n=p_1\cdots p_{2k}$. From \eqref{Mobius inversion formula2} and the fact that if $d|n$, then 
		$\mu(n/d)=\mu(d)$, we obtain
		\begin{equation} \label{eq:mobius}
		\Phi_n(x) \prod_{\substack{d \mid n\\ \mu(d) = -1}} (1-x^d) = 
		\prod_{\substack{d \mid n\\ \mu(d) = 1}} (1-x^d).
		\end{equation}
		Recall that $\Phi_n(x) = \mathrm{P}_S(x)=(1-x) \mathrm{H}_S(x)$. On dividing both sides of \eqref{eq:mobius} by $1-x$ and reducing the resulting identity modulo $x^{p_2+1}$, we find that
		\begin{equation*}
		(1 - x^{p_1}) (1 - x^{p_2}) \mathrm{H}_S(x) \equiv 1\,\,({\rm mod~}x^{p_2+1}),
		\end{equation*}
		which can be rewritten as
		\begin{equation*}
		\mathrm{H}_S(x) \equiv 1 + x^{p_1}\mathrm{H}_S(x) + x^{p_2}\mathrm{H}_S(x)\,\,({\rm mod~}x^{p_2+1}).
		\end{equation*}
		We  deduce that both $p_1$ and $p_2$ are in  $S$ and so $S = \langle p_1, p_2\rangle$ by Lemma~\ref{lem:twoprimes}.
	\end{proof}
	\begin{corollary}[Part 1 of Theorem~\ref{thm:main}]
		\label{onefactorclassification}
		A cyclotomic numerical semigroup $S$ has irreducible semigroup polynomial if and only if  for some distinct primes $p$ and $q$ we have $S=\langle p,q\rangle$ (and so $\mathrm{P}_S=\Phi_{pq}$).
	\end{corollary}

	\section{Cyclotomic numerical semigroups of polynomial length 2} \label{sec:length:2}

	In this section we classify the cyclotomic numerical semigroups with polynomial length $2$, as it was announced in part 2 of Theorem~\ref{thm:main}. As a consequence of this result, it follows that every cyclotomic numerical semigroup with polynomial length $2$ is complete intersection. Our proof of part 2 of Theorem~\ref{thm:main} uses the following three lemmas. 
	
	\begin{lemma}\label{cs-deriv-0}
		Let $S$ be a cyclotomic numerical semigroup. Hence $$\mathrm{P}_S(x) = \prod_{d \in \mathcal{D}} \Phi_d(x)^{f_d}$$ for some finite set of positive integers $\mathcal{D}$ and positive integers $f_d$. Then we have
		\begin{equation*}
		\sum_{d\in {\mathcal D}}f_d\,\mu(d) = 1. 
		\end{equation*}
		In particular, there exists an integer $d>1$ such that $\mu(d) = 1$ and  $\Phi_d \mid \mathrm{P}_S$.
	\end{lemma}
	\begin{proof}
		Since $\mathrm{P}_S(1) = 1$, we have $1 \not \in \mathcal{D}$. 
		In view of \eqref{eq:first-cyclo-coeff}, we obtain
		\[ \mathrm{P}_S(x) \equiv 1-x\sum_{d\in \mathcal D}f_d\,\mu(d) \,\,({\rm mod ~}x^2).\]
		Recalling that $\mathrm{P}_S(x) \equiv 1 - x\,({\rm mod~}x^2)$, it follows that
		$1=\sum_{d\in {\mathcal D}}f_d\,\mu(d)$, 
		and also that there must be some integer $d>1$ in $\mathcal{D}$ with $\mu(d) = 1$.
	\end{proof}

	\begin{lemma} \label{lem:length-2:q}
		Let $S$ be a numerical semigroup such that 
		\begin{equation}
		\label{startingpoint}
		\mathrm{P}_S(x) = \Phi_n(x) f(x^q) 
		\end{equation}
		for some integers $n,q>1$ such that $\mu(n) = 1$ and $f(x)\in \mathbb Z[x]$ is of positive degree. Then $q$ is a prime number and $n = p q$ for some other prime $p \in S$.
	\end{lemma}
	\begin{proof}
		Since $\mu(n)=1$ by assumption, we can write 
		$n=p_1\cdots p_{2k}$ with $p_1 < p_2 < \dots < p_{2k}$ primes. Using \eqref{Mobius inversion formula2} and reducing the resulting expression modulo $x^{p_2 + 1}$ we obtain from \eqref{startingpoint}
		$$
		\mathrm{H}_S(x) (1 - x^{p_1})(1 - x^{p_2}) \equiv f(x^q)\,\,({\rm mod~}x^{p_2+1}),
		$$
		which can be rewritten as
		\begin{equation} \label{eq:l2:1}
		\mathrm{H}_S(x) \equiv f(x^q) + x^{p_1}\mathrm{H}_S(x) + x^{p_2}\mathrm{H}_S(x)\,\,({\rm mod~}x^{p_2+1}).
		\end{equation}
		Since $p_1$ and $p_2$ can not belong to $S$ at the same time by Lemma~\ref{lem:twoprimes}, we see that $f(x^q)$ contains a monomial with exponent $p_1$ or $p_2$. Consequently, $q$ divides $p_1$ or $p_2$, that is, $q \in \{p_1, p_2\}$. Furthermore, if $p \in \{p_1, p_2\} \setminus \{q\}$, then $q$ does not divide $p$ and we find that $p \in S$ by \eqref{eq:l2:1}. Note that 
		$\{p_1, p_2\} = \{p, q\}.$ 
		
		It remains to show that $n=p_1 p_2=pq$. In order to obtain a contradiction we assume that $k > 1$.
		Using \eqref{Mobius inversion formula2} and reducing the resulting expression modulo $x^{p_3 + 1},$ we obtain from \eqref{startingpoint}
		\begin{equation*}
		\mathrm{H}_S(x) (1 - x^{p})(1 - x^q)(1 - x^{p_3}) \equiv f(x^q) (1 - x^{pq})\,\,({\rm mod~}x^{p_3+1}),
		\end{equation*}
		which can be simplified to
		\begin{equation} \label{eq:l2:2}
		\mathrm{H}_S(x) \equiv x^{p}\mathrm{H}_S(x) + x^{p_3} \mathrm{H}_S(x)  + f(x^q) \frac{(1-x^{pq})}{(1-x^q)} \! \pmod {x^{p_3 + 1}}.
		\end{equation}
		Note that 
		\begin{equation*}
		f(x^q) \frac{(1-x^{pq})}{(1-x^q)} = g(x^q),
		\end{equation*}
		for some $g(x)\in \mathbb Z[x].$
		Since $q$ does not divide $p_3$, we must have $p_3 \in S$ in view of \eqref{eq:l2:2}. Since $p \in S$ and $p \cdot p_3$ divides $n$, we conclude by Lemma~\ref{lem:twoprimes} that $S = \langle p,p_3 \rangle$ and hence $\mathrm{P}_S=\Phi_{pp_3}$ is irreducible, whereas by assumption it has at least two irreducible factors.
	\end{proof}
	\begin{lemma} \label{lem:length-2:classification}
		Let $S$ be a numerical semigroup such that 
		\begin{equation*}
		\mathrm{P}_S(x) = \Phi_{pq}(x) \Phi_{l}(x^q)
		\end{equation*}
		for some distinct prime numbers $p$ and $q$, and with $l$ a multiple of $q$. Then either
		\begin{enumerate}
			\item $S=\langle p,q^2\rangle$ and $l = pq$; or
			\item $S=\langle p, q^2,qr\rangle$ for some prime number $r$ such that $r \not \in \{p, q\}$ and $p \in \langle q, r \rangle$, and  $l= qr$.
		\end{enumerate}
	\end{lemma}
	\begin{proof}
		We can write 
		\begin{equation} \label{eq:l2:3}
		\mathrm{H}_S(x) = \Phi_l(x^q) \frac{1 - x^{pq}}{(1 - x^q) (1 - x^p)} = \frac{\Phi_l(x^q)}{1 - x^q} \sum_{i = 0}^{q-1} x^{ip}.
		\end{equation}
		Note that $\Phi_l(x) / (1 -x)\in \mathbb{Z}[\![x]\!].$ We write it
		as $\sum_{j = 0}^\infty a_j x^j.$ We consider $S' = \{j \in \mathbb{N} : a_j \ne 0\}$. 
		Since $\Phi_l(0) = 1$, it follows that $a_0=1$ and hence $0 \in S'$. We are going to prove that $S'$ is a numerical semigroup with $\mathrm{P}_{S'} = \Phi_l$. Equation \eqref{eq:l2:3} can be rewritten in $\mathbb{Z}[\![x]\!]$ as
		\begin{equation} \label{eq:l2:4}
		\mathrm{H}_S(x) = \left(\sum_{j = 0}^\infty a_j x^{jq}\right) \left(\sum_{i = 0}^{q-1} x^{ip}\right)= \sum_{j = 0}^\infty a_j \sum_{i = 0}^{q-1} x^{jq+ip}.
		\end{equation}
		Since $p$ and $q$ are prime numbers, if $jq+ip = j'q+i'p$ with $j, j' \in \mathbb{N}$ and $i, i' \in \{0, 1, \ldots, q-1\}$, then $i = i'$ and $j = j'$. Consequently, $a_j$ is the $qj$-th coefficient of $\mathrm{H}_S$ and hence $a_j \in \{0,1\}$. Note that $j \in S'$ if and only if $qj \in S$. Since $S$ is a numerical semigroup, we see that $S'$ is closed under addition. Furthermore, we have $(1-x)\mathrm{H}_{S'}(x) = \Phi_l(x)$. As a consequence, $S'$ is a numerical semigroup with polynomial $\Phi_l$. By Theorem~\ref{thm:l1} and the assumption $q\mid l,$ we obtain $l = q r$, where $r$ is a prime number different from $q$, and $S' = \langle q,r\rangle$. Summarizing, we have $\mathrm{P}_S(x) = \Phi_{pq}(x) \Phi_{q r}(x^q) =  \Phi_{pq}(x) \Phi_{q^2 r}(x)$, where we used \eqref{eq:cyc property}, and $\mathrm{H}_S(x) = \mathrm{H}_{S'}(x^q) \sum_{i = 0}^{q-1} x^{ip}$. From the latter equality, we obtain $S = qS' + p \mathbb{N}$. There are two possibilities:
		\begin{itemize}
			\item $r = p$. Then $\mathrm{P}_S = \Phi_{pq} \Phi_{pq^2} = \mathrm{P}_{\langle p, q^2 \rangle}$, where in the latter equality we used \eqref{e=2factorincyclotomics}. That is, $S = \langle p, q^2 \rangle$.
			\item $r \ne p$. Then $S = q\langle r,q\rangle +_{q p} p \mathbb{N}$ is a gluing and, thus, $S = \langle p, q^2, qr \rangle$. \qedhere
		\end{itemize}
	\end{proof}
	
	\begin{proof}[Proof of part 2 of Theorem~\ref{thm:main}]
		By Lemma~\ref{cs-deriv-0}, our assumption $\ell(S)=2$ implies that we can write 
		$
		\mathrm{P}_S = \Phi_n \Phi_m\text{~with~}\mu(n) = 1\text{~and~} \mu(m) = 0.
		$
		Let $q$ be the smallest prime such that $q^2\mid m$ and put $l=m/q.$ On applying 
		\eqref{eq:cyc property} we find
		\begin{equation}
		\mathrm{P}_S(x) = \Phi_n(x) \Phi_l(x^{q}).
		\end{equation} 
		Note that $n,l>1.$
		By Lemma~\ref{lem:length-2:q} it now follows that $\mathrm{P}_S(x) = \Phi_{pq}(x) \Phi_l(x^{q})$ for some prime $p\ne q.$
		The proof is completed on invoking Lemma~\ref{lem:length-2:classification} (note that $q\mid l$).
	\end{proof}

	\section{Some interconnected conjectures} \label{sec:open-problems}
	
	\subsection{Polynomial length of cyclotomic numerical semigroups} \label{sec:length}
	
	Let $S$ be a numerical semigroup. Recall that we define the \emph{polynomial length} $\ell(S)$ of $S$ as the number of irreducible factors of $\mathrm{P}_S$ (with multiplicity). If $q>p$ are two primes we have $\mathrm{P}_{\langle p, q \rangle}(x) = \Phi_{pq}(x)$ by \eqref{folklore result} and hence the polynomial length of $\langle p, q \rangle$ is $1$. This observation is generalized in Example~\ref{lengte}. This example involves some more notation that we now
	introduce. Let $d(n)=\sum_{d|n}1$ denote the number of positive divisors of $n$. We have $d(ab) \le  d(a)d(b)$ with equality if $a$ and $b$ are coprime.  
	By ir$(f)$ we denote the number of irreducible prime factors of $f,$ 
	and so $\ell(S)=\text{ir}(\mathrm{P}_S).$
	A  fundamental observation we will use is 
	that
	\begin{equation}
	\label{ir}
	\text{ir}(x^n-1)=d(n)
	\end{equation}
	(this is a consequence of \eqref{xn-1} and the irreducibility of cyclotomic polynomials). Using this and \eqref{Mobius inversion formula} we obtain the (known) identity
	$$\text{ir}(\Phi_n)=1=\sum_{\delta\mid n}d(\delta)\,\mu(n/\delta).$$
	If $S$ is complete intersection with minimal generators $n_1, \ldots, n_e$, and Betti elements $b_1, \ldots, b_{e-1}$ (with multiplicity), then from Corollary~\ref{complete intersection hilbert series} and \eqref{ir} we have
	\begin{equation} \label{eq:length-ci}
	\ell(S)= \sum_{j = 1}^{e-1} d(b_j) - \sum_{j = 1}^{e} d(n_j) + 1.
	\end{equation}
	\begin{example}
		\label{lengte}
		Let $b>a>1$ be coprime integers. The only Betti element of $\langle a, b \rangle$ is $ab$, see \eqref{embedis2}, so we have 
		$$\mathrm{ir}(\mathrm{P}_{\langle a,b\rangle}(x))= d(ab) - d(a) - d(b) + 1.$$
		From \eqref{embedis2} and the multiplicativity of the sum of divisors function $d$ we find that
		$$l(\langle a,b\rangle)=(d(a)-1)(d(b)-1).$$
	\end{example}
	If $S = a_1 S_1 +_{a_1a_2} a_2 S_2$, then from 
	Proposition~\ref{gluing iff Hilbert} and the latter example, we obtain the inequality
	\begin{equation}\label{length gluing}
	\ell(S) \geq \ell(S_1) + \ell(S_2) + (d(a_1)-1)(d(a_2)-1).
	\end{equation}

	\begin{remark}
		This lower bound is sharp. Let $S_1 = \left<2, 3 \right>$ and $S_2 = \left<5, 7 \right>$. Recall that $S_1$ and $S_2$ have length $1$ (Corollary~\ref{onefactorclassification}). We consider the following gluing of $S_1$ and $S_2$, $S = 12 \, \langle 2, 3 \rangle + 5 \, \langle 5, 7 \rangle$. Then the polynomial of $S$ is $\mathrm{P}_{\langle 12, 5 \rangle}(x) \Phi_6(x^{12}) \Phi_{35}(x^5)$. Applying \eqref{eq:cyc property} with $r=5$ and $n=7$ we see that $\Phi_{35}(x^5)=\Phi_{175}(x)$. Further, by repeated application of \eqref{eq:cyc property} we conclude that $\Phi_6(x^{12})=\Phi_6((x^4)^3)=\Phi_{18}(x^4)=\Phi_{18}((x^2)^2)=\Phi_{36}(x^2)=\Phi_{72}(x).$ It follows that
		\begin{equation*}
		\mathrm{P}_S(x)  = \mathrm{P}_{\langle 12, 5 \rangle}(x) \Phi_{72}(x) \Phi_{175}(x),
		\end{equation*}
		and $\ell(S) = (d(12)-1)(d(5) - 1) + \ell(S_1) + \ell(S_2) = 7$.
	\end{remark}
	
	\begin{proposition} \label{prop:length-inequality:ci}
		If $S$ is a complete intersection numerical semigroup, then we have $\mathrm{e}(S) \le \ell(S) + 1$.
	\end{proposition}
	\begin{proof}
		We proceed by induction on $\mathrm{e}(S)$. If $\mathrm{e}(S) = 2$, then the result is trivial. Now let us assume that the result is true for every numerical semigroup with embedding dimension smaller than $\mathrm{e}(S) > 2$. The numerical 
		semigroup $S$ is a gluing of two complete intersection numerical semigroups $S_1$ and $S_2$ by Theorem~\ref{gluing ci}. From \eqref{emdim gluing}, our induction hypothesis, and  \eqref{length gluing} we have
		\[ \mathrm{e}(S) = \mathrm{e}(S_1) + \mathrm{e}(S_2) \leq \ell(S_1) + \ell(S_2) + 2 \leq \ell(S) + 1. \qedhere \]
	\end{proof}

	The next result shows that the inequality in Proposition~\ref{prop:length-inequality:ci} is sharp.
	\begin{proposition} \label{prop:le}
		Let $e \geq 2$ be an integer. For every $l \ge e-1$ there exists a complete intersection numerical semigroup $S$ such that $\ell(S) = l$ and $\mathrm{e}(S) =\nolinebreak e$.
	\end{proposition}
	\begin{proof}
		For every $k \geq 1$, we inductively construct a family of numerical semigroups $S^{(e)}_k$, such that $e(S^{(e)}_k) = e$ and $\ell(S^{(e)}_k) = e+k-2$, as follows:
		\begin{itemize}
			\item $S^{(2)}_k = \langle p_1^k,p_2 \rangle$ for some distinct primes $p_1$ and $p_2$;
			\item $S^{(e+1)}_k = p_1 S^{(e)}_k +_{p_{e+1}p_1} p_{e+1} \mathbb{N}$ for some prime $p_{e+1} \in S^{(e)}_k$ that is not a minimal generator.
		\end{itemize}
		By \eqref{emdim gluing} we conclude that
		\[ e(S^{(e)}_k) = e(S^{(e-1)}_k) + 1 = \dots = e(S^{(2)}_k) + e-2 = e. \]
		By Proposition~\ref{gluing iff Hilbert} we have $\mathrm{P}_{S^{(e)}_k}(x) =  \mathrm{P}_{S^{(e-1)}_k}(x^{p_1})\, \Phi_{p_1p_e}(x)$. Applying this formula recursively we obtain
		\[ \mathrm{P}_{S^{(e)}_k}(x) = \mathrm{P}_{S^{(2)}_k}( x^{p_1^{e-2}} )\,  \prod_{i=3}^{e} \Phi_{p_1p_i} ( x^{p_1^{e-i}} ). \]
		Now, inserting $\mathrm{P}_{S^{(2)}_k}(x) = \prod_{j=1}^k \Phi_{p_1^jp_2}(x),$ which 
		follows by \eqref{e=2factorincyclotomics}, and applying \eqref{eq:cyc property}, we infer that 
		\[ \mathrm{P}_{S^{(e)}_k}(x) = \prod_{i=3}^{e} \Phi_{p_1^{e-i+1}p_i}(x) \,\prod_{j=1}^k \Phi_{p_1^{e+j-2} p_2}(x) , \]
		and hence $ \ell(S_k^{(e)}) = e+k-2$.
	\end{proof}
	
	\begin{conjecture} \label{con:length-inequality}
		Let $S$ be a cyclotomic numerical semigroup. Then $$\mathrm{e}(S) \le \ell(S) + 1.$$
	\end{conjecture}
	
	Using Proposition~\ref{prop:length-inequality:ci} we see that
	\[\text{Conjecture } \ref{conjecture}\Longrightarrow \text{Conjecture }  \ref{con:length-inequality}.\]
	Assuming Conjecture~\ref{con:length-inequality} holds true, the proof of Theorem~\ref{thm:main} can be greatly simplified. Namely, by assumption $\ell(S) \le 2$, so we would have $\mathrm{e}(S) \le \ell(S) + 1 \le 3$ and hence $S$ would be  complete intersection by Corollary~\ref{lessthanfour}.
	
	To conclude this section, we have computed all the cyclotomic numerical semigroups with Frobenius number at most $70$, and classified them in terms of their polynomial length. These cyclotomic numerical semigroups are complete intersections, as it was computationally checked in \cite{ciolan2016cyclotomic}, and there are $835$ of them. The results are displayed in Table~\ref{tab:lenght}. The largest polynomial length found among these semigroups is $8$. Recall that Sections~\ref{sec:length:1} and \ref{sec:length:2} of the present paper study cyclotomic numerical semigroups of polynomial length at most $2$, which add up to $138$ semigroups out of the $835$ computed.
	
	\begin{table}[h]
		\captionsetup{width=0.8\textwidth}
		\centering
		\caption{Number of cyclotomic numerical semigroups with Frobenius number at most $70$, grouped by their polynomial length.}
		\begin{tabular}{l|llllllll}
			Length & 1 & 2 & 3 & 4 & 5 & 6 & 7 & 8 \\
			\hline
			Number of semigroups & 33 & 105 & 224 & 196 & 165 & 74 & 34 & 4
		\end{tabular}
		\label{tab:lenght}
	\end{table}

	\subsection{Cyclotomic exponent sequence}
	Given $f(x)\in {\mathbb Z}[x]$ with $f(0)=1$, there
	exist unique integers $f_j$ such that the formal identity
	$
	f(x)= \prod_{j = 1}^\infty \left(1 - x^j\right)^{f_j}
	$ holds, see  \cite[Lemma 1]{automata} or \cite[Lemma 3.1]{cyclotomic-exponent-sequences}.
	In particular, since $\mathrm{P}_S(0) = 1$, there exist unique integers $e_j$ such that the formal identity
	\begin{equation} \label{eq:exponent-sequence}
	\mathrm{P}_S(x)= \prod_{j = 1}^\infty \left(1 - x^j\right)^{e_j}
	\end{equation}
	holds. The sequence $\mathbf{e}=\{e_j\}_{j \in \mathbb{N}}$ is known as the \emph{cyclotomic exponent sequence} of $S$. This sequence was introduced in \cite[Section 6]{ciolan2016cyclotomic} and later studied in \cite{cyclotomic-exponent-sequences}. From \eqref{eq:exponent-sequence} and the uniqueness of the exponents $\mathbf{e}$, one  can show that $S$ is a cyclotomic numerical semigroup if and only if $\mathbf{e}$ has only a finite number of non-zero elements, see \cite[Proposition 2.4]{cyclotomic-exponent-sequences} for details. If this is the case, then, by \eqref{ir}, we obtain
	\begin{equation*}
	\ell(S)= \sum_{j = 1}^{\infty} e_j \,d(j).
	\end{equation*}
	This equality generalizes equation \eqref{eq:length-ci}, which gives the length of a complete intersection numerical semigroup. 
	
	One of the main results of \cite{cyclotomic-exponent-sequences} is the following.
	
	\begin{theorem}[{\cite[Theorem 1.1]{cyclotomic-exponent-sequences}}] \label{thm:ces:generators}
		Let $S\ne\mathbb N$ be a numerical semigroup and let $\mathbf{e}$ be its cyclotomic exponent sequence. Then
		\begin{enumerate}
			\item $e_1 = 1;$
			\item $e_j = 0$ for every $j \ge 2$ not in $S;$
			\item $e_j = -1$ for every minimal generator $j$ of $S$;
			\item $e_j = 0$ for every $j$ in $S$ that has only one factorization and is not a minimal generator.
		\end{enumerate}
	\end{theorem}
	As a consequence of this theorem, the set $\{n \in \mathbb{N}: e_n < 0\}$ is a system of generators of $S$. In the case of cyclotomic numerical semigroups, the authors of \cite{cyclotomic-exponent-sequences} made the following conjecture.
	
	\begin{conjecture}[{\cite[Conjecture 7.1]{cyclotomic-exponent-sequences}}] \label{con:msg}
		Let $S$ be a cyclotomic numerical semigroup and let $\mathbf{e}$ be its cyclotomic exponent sequence. Then $n \in \mathbb{N}$ is a minimal generator of $S$ if and only if $e_n < 0$.
	\end{conjecture}
	
	The cyclotomic exponent sequence of a complete intersection numerical semigroup can be easily obtained from Corollary~\ref{complete intersection hilbert series}. Note that for these semigroups the only integers with $e_n < 0$ are the minimal generators. Hence, as already noted in \cite[Proposition 7.3]{cyclotomic-exponent-sequences}, we have 
	\[\text{Conjecture } \ref{conjecture}\Longrightarrow \text{Conjecture }  \ref{con:msg}.\]
	In the rest of this section we relate Conjectures~\ref{con:length-inequality} and \ref{con:msg}. In order to do so, we formulate a further conjecture.
	
	\begin{conjecture} \label{con:maximals-D}
		Let $S$ be a cyclotomic numerical semigroup. Hence $$\mathrm{P}_S(x) = \prod_{d \in \mathcal{D}} \Phi_d(x)^{f_d}$$ for some finite set of positive integers $\mathcal{D}$ and positive integers $f_d$. Let $\mathbf{e}$ the cyclotomic exponent sequence of $S$. 
		Then 
		$$\{d \ge 2 : e_d > 0\} \subseteq \mathcal{D},$$ and $e_d \le f_d$ for every $d \ge 2$ with $e_d > 0$.
	\end{conjecture}
	
	In addition, we will need the following proposition, the proof of which we include for completeness.
	
	\begin{proposition}[{\cite[Proposition 2.3]{cyclotomic-exponent-sequences}}] \label{prop:sum-e}
		Let $S$ be a numerical semigroup and let $\mathbf{e}$ be its cyclotomic exponent sequence. If $S$ is cyclotomic, then 
		$\sum_{j \ge 1} e_j  = 0$.
	\end{proposition}
	\begin{proof}
		Let $N$ be the largest index $j$ such that $e_j\ne 0.$ Then we have
		\[ \operatorname{P}_S(x)= (1-x)^{\sum_{j\le N}e_j}G_S(x),\] 
		for some rational function $G_S(x)$ satisfying $G_S(1)\not \in \{ 0, \infty \}$
		(in fact $G_S(1)=\prod_{j\le N}j^{e_j}$). Since $\operatorname{P}_S(1)=1$, it follows that $\sum_{j\ge 1}e_j=0$. 
	\end{proof}
	
	\begin{proposition}
		The following implications hold:
		\[
		\text{Conjecture \ref{conjecture}} \Longrightarrow \text{Conjecture \ref{con:msg}} \Longrightarrow \text{Conjecture \ref{con:maximals-D}} \Longrightarrow \text{Conjecture \ref{con:length-inequality}.}
		\]
	\end{proposition}
	\begin{proof} 
		The first implication follows from Corollary~\ref{complete intersection hilbert series}.  Let $S$ be a cyclotomic numerical semigroup and let $n_1, \ldots, n_e$ be its minimal generators. Let us assume that $S$ satisfies Conjecture~\ref{con:msg}. Then we have
		\begin{equation*}
		\prod_{i = 1}^e (1-x^{n_i}) \,\mathrm{P}_S(x) = \prod_{d \in \mathbb{N}; \, e_d > 0} (1 - x^d)^{e_d}
		\end{equation*}
		Let $d \in \mathbb{N}$ with $e_d > 0$ and $d \ge 2$. We are going to prove that $\Phi_d^{e_d}$ divides $\mathrm{P}_S$. Recall that $\Phi_d^{e_d}$ divides $(1-x^d)^{e_d}$ exactly. Now we argue that $\Phi_d$ does not divide $\prod_{i = 1}^e (1-x^{n_i})$. Note that $\Phi_d$ divides $\prod_{i = 1}^e (1-x^{n_i})$ if and only if there exists $i \in \{1, 2, \ldots, e\}$ such that $d$ divides $n_i$.    
		By Theorem~\ref{thm:ces:generators} it follows that $d \ne n_i$ and  $d \in S.$ Since by assumption $n_i$ is a minimal generator of $S$, we conclude that $d$ does not divide $n_i$. This along with the fact that $\Phi_d$ is irreducible allows us to conclude that $\Phi_d^{e_d}$ exactly divides $\mathrm{P}_S$ and, hence, $e_d \le f_d$.
		
		Now let us assume that $S$ satisfies Conjecture~\ref{con:maximals-D}. Since $\{d \in \mathbb{N} : e_d < 0\}$ is a (finite) system of generators of $S$ by Theorem~\ref{thm:ces:generators}, we find that
		\[\mathrm{e}(S) \le \sum_{d \in \mathbb{N}; \, e_d < 0} (-e_d).\]
		From Proposition~\ref{prop:sum-e}, we obtain
		\[ \sum_{d \in \mathbb{N}; \, e_d < 0} (-e_d) = \sum_{d \in \mathbb{N}; \, e_d > 0} e_d. \]
		Finally, because we are assuming that $S$ satisfies Conjecture~\ref{con:maximals-D}, we have
		\[\mathrm{e}(S)\le \sum_{d \in \mathbb{N}; \, e_d > 0} e_d \le e_1 + \sum_{d \in \mathcal{D}} f_d = 1+\ell(S). \]
		We conclude that $\mathrm{e}(S) \le \ell(S) + 1$.
	\end{proof}

	\noindent \textbf{Acknowledgement.} The authors thank Pedro A. Garc\'ia-S\'anchez for putting the authors in contact with each other and for helpful conversations on this work. In addition they thank the referee for helpful feedback.
	
	\bibliographystyle{abbrv}
	\bibliography{Reference.bib}
	
	\noindent
	{\scshape Alessio Borz\`{i}} \quad \texttt{Alessio.Borzi@warwick.ac.uk}\\
	{\scshape  Mathematics Institute, University of Warwick, Coventry CV4 7AL, United Kingdom.}
	
	\medskip
	\noindent
	{\scshape Andr\'es Herrera-Poyatos} \quad \texttt{andres.herrerapoyatos@cs.ox.ac.uk}\\
	{\scshape Department of Computer Science, University of Oxford, Wolfson Building, Parks Road, Oxford, OX1 3QD, United Kingdom.}
	
	\medskip
	\noindent
	{\scshape Pieter Moree} \quad \texttt{moree@mpim-bonn.mpg.de}\\
	{\scshape Max-Planck-Institut f\"ur Mathematik, Vivatsgasse 7, D-53111 Bonn, Germany.}
	
\end{document}